\documentclass{amsart}
\usepackage{amsmath}
\usepackage{amsxtra}
\usepackage{amstext}
\usepackage{amssymb}
\usepackage{latexsym}
\usepackage{dsfont} % for \mathds{N}
\usepackage{verbatim}
\usepackage{tabls}
\usepackage{rotating}
\usepackage{color}
\newtheorem{thm}{Theorem}[section]

\newtheorem{prop}[thm]{Proposition}
\theoremstyle{definition}
\newtheorem{defn}[thm]{Definition}

\newtheorem{rem}[thm]{Remark}

\numberwithin{equation}{section}

\def\Xint#1{\mathchoice
   {\XXint\displaystyle\textstyle{#1}}%
   {\XXint\textstyle\scriptstyle{#1}}%
   {\XXint\scriptstyle\scriptscriptstyle{#1}}%
   {\XXint\scriptscriptstyle\scriptscriptstyle{#1}}%
   \!\int}
\def\XXint#1#2#3{{\setbox0=\hbox{$#1{#2#3}{\int}$}
     \vcenter{\hbox{$#2#3$}}\kern-.5\wd0}}

\def\dashint{\Xint-}

\begin{document}
\title[Nonlinear potentials]{On the good-$\lambda$ inequality for nonlinear potentials}
\date{\today}

\author{Petr Honz\'\i k}
\address{Institute of Mathematics, AS CR,  
Zitn\'{a} 25, CZ - 115 67 Praha 1, Czech 
Republic} 
\email{honzik@gmail.com}
\author{Benjamin J. Jaye}
\address{Department of Mathematics, University of Missouri, Columbia, Missouri 65211, USA}
\email{bjjm93@mizzou.edu}



\thanks{Honzik was supported by the Institutional Research Plan no.
AV0Z10190503 and by the grant KJB100190901 GAAV.  Jaye was partially supported by NSF grant DMS-0901550.}

\begin{abstract}
This note concerns an extension of the good-$\lambda$ inequality for fractional integrals, due to B. Muckenhoupt and R. Wheeden \cite{MW}.  The classical result is refined in two aspects.  Firstly, general nonlinear potentials are considered; and secondly, the constant in the inequality is proven to decay exponentially.  As a consequence, the exponential integrability of the gradient of solutions to certain quasilinear elliptic equations is deduced.  This in turn is a consequence of certain Morrey space embeddings which extend classical results for the Riesz potential.  In addition, the good-$\lambda$ inequality proved here provides an elementary proof of the result of Jawerth, Perez and Welland \cite{JPW} regarding the positive cone in certain weighted Triebel-Lizorkin spaces.
\end{abstract}

\maketitle

\section{Introduction}

For $n\geq 1$, $0<\alpha<n$ and $0<q\leq \infty$, we form the nonlinear potential $\mathbf{T}^q_{\alpha}(\mu)$ of a non-negative measure $\mu$, by:
\begin{equation}\label{contpot}
\mathbf{T}(\mu)(x) = \mathbf{T}^q_{\alpha}(\mu)(x) = \Bigl(\int_0^{\infty}\Bigl(\frac{\mu(B(x,r))}{r^{n-\alpha}}\Bigl)^{q}\frac{dr}{r}\Bigl)^{1/q}.
\end{equation}
Whenever there is no confusion we will omit the indices in the definition of $\mathbf{T}$.   In the case that $q=1$, we recover (up to a constant) the classical Riesz potential $\mathbf{I}_{\alpha}(\mu)$, and if $q=\infty$, the resulting operator is the fractional maximal operator $\mathbf{M}_{\alpha}(\mu)$.

For general $0<q<\infty$, $\mathbf{T}(\mu)$ can be written as a power of T. Wolff's potential; and (despite its many indices) plays an important role in the study of Sobolev and Triebel-Lizorkin spaces \cite{AH, HW, JPW}, as well as in the theory of certain quasi-linear equations \cite{KM, DM2, DM3}.

It is well known that although the potential operator $\mathbf{T}^{q}_{\alpha}(d\mu)$ is not in general pointwise comparable to the fractional maximal operator $\mathbf{M}_{\alpha}(d\mu)$; a useful substitute is the 
%principle that one can use a smallness assumption on the fractional maximal operator to obtain quantitative decay estimates on the distribution function of the potential operator.  This fundamental result was proved by Muckenhoupt and Wheeden \cite{MW}, and is known as the 
distributional \textit{good-$\lambda$ inequality} of Muckenhoupt and Wheeden \cite{MW}.  Here we consider an extension and sharpening of the result of \cite{MW}, along with applications to certain exponential integrability results.  Our primary theorem is the following (for the definition of weak $A_{\infty}$ weights, see Definition \ref{ainfdefin} below):

\begin{thm}\label{contthm} Let  $0<\alpha<n$ and $0<q<\infty$, and suppose that $\sigma$ is a weak $A_{\infty}$ weight.  Then there exist constants $\tau>1$ and $C,\, c>0$ depending on $n, \,q,\,\alpha$ and the weak $A_{\infty}$ character of $\sigma$, such that for each $\lambda>0$ and $0<\epsilon<1$, it follows:
\begin{equation}\label{contlambda}\begin{split}
|\{x\in \mathbf{R}^n\, :\, \mathbf{T}(\mu)(x) > \tau\lambda, & \; \mathbf{M}_{\alpha}(\mu)(x) \leq \epsilon \lambda\}|_{\sigma}\\
& \leq C e^{-c/\epsilon^q}|\{x\in \mathbf{R}^n\, :\, \mathbf{T}(\mu)(x) > \lambda\}|_{\sigma}.
\end{split}\end{equation}
Here $|E|_{\sigma}$ denotes the $\sigma$ measure of a set $E$.
\end{thm}

The main observation of this paper is the growth of the display (\ref{contlambda}) in the parameter $\epsilon$.  Exploiting this improvement in the constant allows us to deduce certain exponential integrability results for nonlinear potentials (Theorem \ref{morremb} below).  In the classical inequality of \cite{MW}, one replaces the exponential constant in (\ref{contlambda}) with the polynomial factor $\epsilon^{n/(n-\alpha)}$.  It is shown below that the inequality (\ref{contlambda}) is sharp in terms of the power of $\epsilon$, see Proposition \ref{sharpprop}.  The exact behavior of the parameter $\epsilon$ in such distributional inequalities is topic which has attracted the interest of several authors, see e.g. Banuelos \cite{B1}; Chang, Wilson and Wolff \cite{CWW} and references therein.

To illustrate the principle ideas in the proof of Theorem \ref{contthm}, we present the proof first for dyadic operators in the unweighted case.  In this set-up, one can give a simple upper bound on the exponential constant in (\ref{contthm}) explicitly, see Theorem \ref{mainthm} below.    

Theorem \ref{contthm} provides an elementary treatement of several useful inequalities in harmonic analysis and potential theory.    Indeed, from integrating out the good-$\lambda$ inequality in Theorem \ref{contthm} in the standard fashion, see e.g. \cite{MW, AH}, one readily recovers the inequality of Jawerth, Perez and Welland \cite{JPW}.  This states that for each weak $A_{\infty}$ weight $\sigma$, and for indices $0<p<\infty$, $0<\alpha<n$ and $0<q<\infty$, it follows that:
\begin{equation}\label{liztri}\Bigl(\int_{\mathbf{R}^n}(\mathbf{T}^q_{\alpha}(\mu)(x))^pd\sigma \Bigl)^{1/p} \leq C(\sigma,q,\alpha, p) \Bigl(\int_{\mathbf{R}^n} (\mathbf{M}_{\alpha}(\mu)(x))^p d\sigma\Bigl)^{1/p},
\end{equation}
for all positive measures $\mu$ (the case $q>1$ is contained in the original result of \cite{MW}).

The inequality (\ref{liztri}) was first proved (for the usual Muckenhoupt $A_{\infty}$ class) in \cite{JPW} by means of local maximal functions.  An alternative approach was found by Cohn and Verbitsky \cite{CV}, which employed techniques developed by Wolff \cite{HW}.

Furthermore, from Theorem \ref{contthm} one can obtain the natural end point of (\ref{liztri}) at $q=\infty$.  This comes courtesy of the exponential nature of the constant in (\ref{contlambda}), and concerns the exponential integrability of $\mathbf{T}(\mu)$, when $\mu$ lies in a suitable local Morrey space.  The following theorem extends to non-linear potential operators certain well known Morrey space embeddings for Riesz potentials due to D. Adams \cite{A1}, who in turn was building on the work of Stampacchia \cite{St1}.

Let us fix the notation that for a set $E\subset\mathbf{R}^n$, we define the restriction of a measure $\mu$ to $E$ by $d\mu_E = \chi_E d\mu.$

\begin{thm}\label{morremb}  Let $\sigma$ be a weak $A_{\infty}$ weight.  There exist finite positive constants $C, c>0$ depending on $n,q,$ $\alpha$ and the weak $A_{\infty}$ character of $\sigma$, so that:
\begin{equation}\label{expint}
\frac{1}{|2B|_{\sigma}}\int_{2B} \exp\Bigl(c\frac{\mathbf{T}^q_{\alpha}(\mu_{B})(x)}{||\mathbf{M}_{\alpha}(\mu_B)||_{L^{\infty}(B)}}\Bigl)^q d\sigma \leq C,
\end{equation}
for all balls $B\subset\mathbf{R}^n$ so that $||\mathbf{M}_{\alpha}(\mu_B)||_{L^{\infty}(B)}<\infty$.
\end{thm}
In display (\ref{expint}) and elsewhere, $2B$ denotes the concentric double of $B$.  The exponent in (\ref{expint}) is sharp, see Remark \ref{expsharp}.  Theorem \ref{morremb} follows from Theorem \ref{contthm} and a simple localization, and is carried out in Section \ref{expsec}.  We remark that certain exponential integrability results for linear integral operators can be obtained via the re-arranged good-$\lambda$ inequality introduced in Bagby and Kurtz \cite{BK} and generalized by Vyb\'{i}ral \cite{Vyb}.

In addition, we offer an alternative proof of Theorem \ref{morremb} that goes via a distributional inequality in the parameter $\tau$ appearing in (\ref{contlambda}).  This may be of independent interest, and is the content of the following result:

\begin{thm}\label{goodtauthm} Let $0<\alpha<n$ and $0<q<\infty$, and suppose that $\sigma$ is a weak $A_{\infty}$ weight.  Then there exist constants $c>1$ and $0<C<1$, depending on $n, \,q,\,\alpha$ and the $A_{\infty}$ character of $\sigma$, such that for each $\lambda>0$ and $0<\epsilon<1$, it follows:
\begin{equation}\label{goodtau}\begin{split}
|\{x\in \mathds{R}^n\, :\, (\mathbf{T}_{\alpha}^q(\mu))^q(x) > (1+c\epsilon)\lambda, & \; (\mathcal{M}_{\alpha}(\mu))^q(x) \leq \epsilon \lambda\}|_{\sigma} \\& \leq C|\{x\in \mathds{R}^n\, :\, (\mathbf{T}_{\alpha}^q(\mu))^q(x) > \lambda\}|_{\sigma}.
\end{split}\end{equation}
\end{thm}

%The use of nonlinear potentials in the inequality (\ref{contlambda}) also provides an elementary proof of the useful fact that the positive cone in certain weighted Triebel-Lizorkin spaces does not depend on the sequence space parameter, a result originally due to Jawerth, Perez and Welland \cite{JPW}.  More precisely, for each weak $A_{\infty}$ weight $\sigma$, and for indices $0<p<\infty$, $0<\alpha<n$ and $0<q<\infty$, it follows that:
%\begin{equation}\label{liztri}\Bigl(\int_{\mathbf{R}^n}(\mathbf{T}^q_{\alpha}(\mu)(x))^pd\sigma \Bigl)^{1/p} \leq C(\sigma,q,\alpha, p) \Bigl(\int_{\mathbf{R}^n} (\mathbf{M}_{\alpha}(\mu)(x))^p d\sigma\Bigl)^{1/p},
%\end{equation}
%for all positive measures $\mu$.  The inequality (\ref{liztri}) was first proved (for the usual Muckenhoupt $A_{\infty}$ class) in \cite{JPW} by means of local maximal functions.  An alternative approach was found by Cohn and Verbitsky \cite{CV}, which employed techniques developed by Wolff \cite{HW}.   Theorem \ref{contthm} yields a third approach, as (\ref{liztri}) is an immediate consequence of integrating out the good-$\lambda$ inequality in the standard fashion, see e.g. \cite{MW, AH}.

Before we describe further consequences of Theorem \ref{contthm} in terms of regularity of solutions to certain elliptic equations, we turn to the definition of weak $A_{\infty}$ weights:

\begin{defn}\label{ainfdefin}A nonnegative measurable function $\sigma$ is said to be a weak $A_{\infty}$weight, if there exist constants $C_{\sigma}>0$ and $\theta>0$ so that for each cube $Q$ and measurable set $E\subset Q$:
\begin{equation}\label{defainf}
\frac{|E|_{\sigma}}{|2Q|_{\sigma}}\leq C_{\sigma}\Bigl(\frac{|E|}{|Q|}\Bigl)^{\theta}.
\end{equation}
A constant is said to depend on the weak $A_{\infty}$ character of $\sigma$ if it depends on $\theta$ and $C_{\sigma}$.

This class of weights was introduced by Sawyer \cite{S1}.  Note that weak $A_{\infty}$ weights may vanish on open sets, and as a result are not necessarily doubling.
\end{defn}

By combining Theorem \ref{morremb} with the deep gradient estimates recently proved by Duzaar and Mingione \cite{DM2, DM3} (see Theorem \ref{gradbounds} below), one can obtain exponential integrability results for the gradient of certain quasilinear equations.  For the sake of brevity, we will only consider a simple model equation, but the same results hold for quasilinear operators in the generality considered in \cite{DM2}.  For $1<p<\infty$, let us define the $p$-Laplacian operator by:
$$\Delta_p u = \text{div}(|\nabla u|^{p-2}\nabla u). 
$$
Let $\Omega$ be a bounded domain, and $\mu$ a finite nonnegative measure defined on $\Omega$.  We consider a class of very weak solutions\footnote{The reader with some exposure of (\ref{dirprob}) can replace this notion of weak solution with their personal favorite, for example \textit{renormalized solutions} \cite{DMMOP}, or \textit{p-superharmonic solutions} see \cite{HKM}.  We make the definition here to reduce technicalities.} $u$ (see Definition \ref{vweak}) of the equation:
\begin{equation}\label{dirprob}-\Delta_p u =\mu \text{ in }\Omega.
\end{equation}
%We may consider any suitable notion of solution of (\ref{dirprob}), for instance solutions in the $p$-superharmonic sense \cite{HKM}, or alternatively renormalized solutions \cite{DMMOP}, this follows from a well know approximation procedure, 

\begin{thm}\label{quasthm}  Suppose $u$ is a very weak solution of (\ref{dirprob}).

\indent (i).  If $2-1/n<p\leq 2$, then there exist constants $c,\,C>0$ depending on $n$ and $p$ so that:
\begin{equation}\label{pleq2bd}
\dashint_B \exp\Bigl(c\frac{|\nabla u|^{p-1}}{||\mathbf{M}_{1}(\mu_{2B})||_{L^{\infty}(2B)}}\Bigl)dx\leq C \exp\Bigl(c\dashint_B\frac{|\nabla u| \, dx}{||\mathbf{M}_{1}(\mu_{2B})||_{L^{\infty}(2B)}^{\frac{1}{p-1}}}\Bigl)^{p-1}
\end{equation}
for all balls $B$ so that $2B\subset\Omega$ and $||\mathbf{M}_{1}(\mu_{2B})||_{L^{\infty}(2B)}<\infty$.

%\begin{equation}\label{pleq2bd}
%\dashint_B \exp\Bigl(c\frac{|\nabla u|^{p-1}}{||\mathbf{M}_{\alpha}(\mu_{2B})||_{L^{\infty}(2B)}}\Bigl)dx\leq C(n,p),
%\end{equation}
\indent (ii). If $p\geq 2$, then there exist constants $C,\, c>0$ depending on $n$ and $p$ so that:
\begin{equation}\label{pgeq2bd}
\dashint_B \exp\Bigl(c\frac{|\nabla u|^{\frac{p}{2}}}{||\mathbf{M}_{1}(\mu_{2B})||_{L^{\infty}(2B)}^{\frac{p}{2(p-1)}}}\Bigl)dx\leq C \exp\Bigl(c\dashint_B\frac{|\nabla u|^{\frac{p}{2}} \,dx}{||\mathbf{M}_{1}(\mu_{2B})||_{L^{\infty}(2B)}^{\frac{p}{2(p-1)}}}\Bigl).
\end{equation}
for all balls $B$ so that $2B\subset\Omega$ and $||\mathbf{M}_{1}(\mu_{2B})||_{L^{\infty}(2B)}<\infty$.
%\begin{equation}\label{pgeq2bd}
%\dashint_B \exp\Bigl(c\frac{|\nabla u|^{p/2}}{||\mathbf{M}_{\alpha}(\mu_{2B})||^{p/2(p-1)}_{L^{\infty}(2B)}}\Bigl)dx\leq C(n,p).
%\end{equation}
\end{thm}

One should compare Theorem \ref{quasthm} with the recent Lipschitz regularity results of Cianchi and Maz'ya \cite{CM}, and Duzaar and Mingione \cite{DM1}.  In these papers, the Lipschitz regularity of a solution $u$ of (\ref{dirprob}) is proved if $d\mu = f(x)dx$ with $f\in L^{n,1}(\Omega)$.  Here $L^{n,1}(\Omega)$ is the Lorentz space.  In particular, the exponential integrability result Theorem \ref{quasthm} holds for all $f\in L^n(\Omega)$.

Using Theorem \ref{morremb}, one can also recover known exponential integrability results for the solution $u$ of (\ref{dirprob}) from the estimates of Kilepl\"{a}inen and Maly \cite{KM}.  The following result (at least when $p\geq 2$) is contained in \cite{Min07}:

\begin{thm} \label{uexpintthm} Let $1<p<n$, and suppose that $u\in W^{1,p-1}_{\text{loc}}(\Omega)$ is a positive solution of (\ref{dirprob}).  Then there exist constants $c, C>0$, depending on $n$ and $p$, so that for all balls $B=B(x,r)$ so that $2B\subset \Omega$:
\begin{equation}\label{uexpint}
\dashint_{B}\exp\Bigl(c\frac{u(x)}{||\mathbf{M}_{p}(\mu_{2B})||_{L^{\infty}(2B)}}\Bigl) dx\leq C\exp\Bigl(c \dashint_B \frac{u}{||\mathbf{M}_{p}(\mu_{2B})||_{L^{\infty}(2B)}}dx\Bigl).
\end{equation}
\end{thm}

Additional applications of Theorem \ref{morremb} to certain quasilinear elliptic equations will be presented in \cite{JV11}.  In these latter applications it will be important that Theorem \ref{morremb} holds for the full class of weak $A_{\infty}$ weights.

The results of this paper are split into two sections.  Firstly, in Section \ref{goodlamsec} we consider the good-$\lambda$ inequalities.  Subsequently, in Section \ref{expsec}, the results concerning exponential integrability and its consequences to quasilinear elliptic equations are proved.

%The structure of this paper is follows.   Firstly, in Section \ref{goodlamsec} we prove Theorem \ref{contthm} together with its dyadic analogue.  We begin with a discussion of dyadic potential operators (Theorem \ref{mainthm}), where the ideas are most transparent.  Subsequently we exhibit a simple example (Proposition \ref{sharpprop}) to show the sharpness of the constant appearing in (\ref{contlambda}), before moving on to proving Theorem \ref{contthm}.  Finally in Section \ref{goodlamsec}, we prove Theorem \ref{goodtauthm}.

%In Section \ref{expsec}, we consider exponential integrability; Theorem \ref{morremb} is proved twice - once as a consequence of Theorem \ref{contthm}, and also as a result of Theorem \ref{goodtauthm}.  We conclude the paper with the applications mentioned in Theorems \ref{quasthm} and \ref{uexpintthm}.

\section{The good-$\lambda$ inequality}\label{goodlamsec}

\subsection{The dyadic inequality}In this section we will prove Theorem \ref{contthm}.  To illustrate the idea of the proof in the simplest case we will work primary with the dyadic analogue of the operator $\mathbf{T}$, which we introduce now.  Let $\mathcal{Q}$ be the lattice of dyadic cubes in $\mathbf{R}^n$.  At each level $k \in \mathds{Z}$, the dyadic cubes of length $2^k$ (which we denote by $\mathcal{Q}_k$) are the collection of cubes in $\mathbf{R}^n$ which are the translations by $2^k\lambda$ for $\lambda = (\lambda_1,...,\lambda_n) \in \mathbf{Z}^n$ of the cube $[0,2^k)^n$.  Then the dyadic lattice $\mathcal{Q}$ is the union of these collections $\mathcal{Q}_k$ over all integers $k\in \mathds{Z}$.

Given $0<\alpha<n$ and $0<q<\infty$, we form the nonlinear homogeneous operator $\mathcal{T}$, acting on a nonnegative Borel measure $\mu$, by:
\begin{equation}\label{tdefn}\mathcal{T}(\mu)(x) = \Bigl\{\sum_{x\in Q\in \mathcal{Q}}\Bigl(\frac{\mu(Q)}{\ell(Q)^{n-\alpha}}\Bigl)^q\Bigl\}^{1/q}.\end{equation}
In the case when $q=\infty$, we denote the resulting operator by $\mathcal{M}_{\alpha}(\mu)$, the dyadic fractional maximal operator.
Our primary result for the dyadic operator is the following, which we state in the unweighted case:
\begin{thm}\label{mainthm} Let  $0<\alpha<n$ and $0<q<\infty$.  Then there exists a positive constant $C$, depending on $n, q,$ and $\alpha$, such that for each $\lambda>0$ and $0<\epsilon<1$, it follows:
\begin{equation}\label{goodlambda}\begin{split}
|\{x\in \mathbf{R}^n\, :\, \mathcal{T}(\mu)(x) > 2\lambda, & \; \mathcal{M}_{\alpha}(\mu)(x)\leq \epsilon \lambda\}| \\
&\leq C2^{-\frac{\alpha}{\epsilon^q}(2^q-1)}|\{x\in \mathbf{R}^n\, :\, \mathcal{T}(\mu)(x) > \lambda\}|.
\end{split}\end{equation}
\end{thm}

\begin{proof}[Proof of Theorem \ref{mainthm}] We will suppose $\mu$ has compact support.  Once (\ref{goodlambda}) has been established for such measures, the theorem follows by a routine approximation.

First denote $G_{\lambda} = \{x\, :\, \mathcal{T}(\mu)(x) > \lambda\}$.  This is an open set by lower semicontinuity of $\mathcal{T}$.  Let us decompose $G_{\lambda} = \cup_j Q_j$ where $Q_j$ are pairwise disjoint maximal dyadic (Whitney) cubes.  Thus, if $Q'_j$ is the dyadic parent cube of $Q_j$, then there is a point $z\in Q'_j$ so that $\mathcal{T}(\mu)(z) \leq \lambda$.

It suffices to prove, for each $j$, that:
\begin{equation}\label{goodlambdalocal}
|\{x\in Q_j\, :\, \mathcal{T}(\mu)(x) > 2\lambda, \; \mathcal{M}_{\alpha}(\mu)(x) \leq \epsilon \lambda\}| \leq Ce^{-c/\epsilon^q}|Q_j|,
\end{equation}
for positive constants $C$ and $c$, depending on $n, q,$ and $\alpha$.  Let us fix such a $Q_j$, and let $j_0$ be such that $\ell(Q_j) = 2^{j_0}$.   Without loss of generality we may suppose there exists $z\in Q_j$ such that $\mathcal{M}_{\alpha}(\mu)(z) \leq \epsilon \lambda$.  Let $S = \{x\in Q_j\, :\, \mathcal{T}(\mu)(x) > 2\lambda, \; \mathcal{M}_{\alpha}(\mu)(x) \leq \epsilon \lambda\}$.

Let us first note that since $Q_j$ is a maximal cube, there is a straightforward estimate on the tail of the potential:
\begin{equation}\label{tailest}\sum_{\substack{Q\supsetneq Q_j\\Q\in \mathcal{Q}}}\Bigl(\frac{\mu(Q)}{\ell(Q)^{n-p}}\Bigl)^{q} \leq \lambda ^q.
\end{equation}
The main technical novelty of our proof in comparison to the classical case is the following step, which uses the smallness of the fractional maximal operator somewhat more efficiently.  For a fixed $k\in \mathds{Z}$, we define $g_k(x)$:
\begin{equation}\label{gdef}g_k(x) = \sum_{\substack{x\in Q\\Q\in \mathcal{Q}_k}}\frac{\mu(Q)}{\ell(Q)^{n-\alpha}}.
\end{equation}
Note in particular that the sum in (\ref{gdef}) is over those cubes of sidelength $2^k$, so there is only one nonzero term in the sum, for any $x$.  Then for any $\gamma>0$, and $k\leq j_0$, it follows that:
\begin{equation}\label{glevel}
|\{x\in Q_j\, : g_{k}(x)>\gamma\}|\leq \frac{\epsilon \lambda}{\gamma} 2^{(k-j_0)\alpha}|Q_j |.
\end{equation}
To see the estimate (\ref{glevel}), note from Chebyshev's inequality:
\begin{equation}\begin{split}\nonumber
|\{x\in Q_j\, : g_{k}(x)>\gamma\}| & \leq \frac{1}{\gamma} \sum_{\substack{Q\in \mathcal{Q}_k \\ Q\subset Q_j}} \frac{\mu(Q)}{\ell(Q)^{n-\alpha}}|Q| \leq\frac{1}{\gamma}2^{k\alpha}\mu(Q_j)\\
& \leq \frac{1}{\gamma} |Q_{j_0}|2^{(k-j_0)\alpha}\epsilon \lambda.
\end{split}\end{equation}
In the last inequality we have used the estimate $\mathcal{M}_{\alpha}(\mu)(z) \leq \epsilon \lambda$ for some $z\in Q_j$.

Let $m\in \mathds{N}$ to be chosen later.  Then, note that if $x\in Q_j$ and $\mathcal{M}_{\alpha}(\mu)(x)\leq \epsilon \lambda$, it follows from (\ref{tailest}) that:
\begin{equation}\begin{split}\label{3split}
\sum_{\substack{x\in Q,\\Q\in \mathcal{Q}}}\Bigl(\frac{\mu(Q)}{\ell(Q)^{n-p}}\Bigl)^{q}  \leq \lambda^q + m(\epsilon\lambda)^q + \sum_{\substack{x\in Q,\, Q\in \mathcal{Q}, \\ \ell(Q) \leq 2^{j_0-m }}}\Bigl(\frac{\mu(Q)}{\ell(Q)^{n-p}}\Bigl)^{q}
\end{split}\end{equation}
Let $\beta<\alpha$, and note that:
$$S\subset Q_j\cap E, \text{ where }E = \bigcup_{k\leq j_0-m}^{\infty}\{g_k\geq \lambda 2^{\beta(k-j_0+m)}((2^q - m\epsilon^q -1) (1-2^{-q\beta}))^{1/q}\}.
$$
This is a simple consequence of (\ref{3split}), together with $\mathcal{T}(f)(x)>2\lambda$ for $x\in S$.
We now use the estimate (\ref{glevel}) to estimate $|S|$.  Indeed:
\begin{equation}\begin{split}\nonumber
|S|& \leq \sum_{k\leq j_0-m}^{\infty}|Q_j\cap\{g_k\geq \lambda 2^{\beta(k-j_0+m)}[(2^q - m\epsilon^q -1)(1-2^{-q\beta})]^{1/q}\}|\\
&\leq \sum_{k\leq j_0-m}^{\infty}|Q_j|\frac{\epsilon}{[(1-2^{-q\beta})(2^q - m\epsilon^q-1)]^{1/q}}2^{\alpha (k-j_0)}2^{\beta (j_0-m-k)}
\end{split}\end{equation}
In conclusion:
$$|S|\leq  2^{-\alpha m}|Q_j|\frac{\epsilon}{[(1-2^{-q\beta})(2^q - m\epsilon^q-1)]^{1/q}} \sum_{k=0}^{\infty} 2^{(\beta-\alpha)k}.
$$
It remains to make a good choice of $m$.  Let $m$ be:
$$m = \lfloor \frac{2^q-1}{\epsilon^q}-1\rfloor.
$$
Then, $|S|\leq C2^{-\alpha m}|Q_j|$, where $C = C(n,q,\alpha)>0$.
\end{proof}

\subsection{On the sharpness of the Theorem \ref{mainthm}} In this section we present an example to exhibit the sharpness of good-$\lambda$ inequality.  The example concerns the linear case $q=1$, but a simple adaptation works for the nonlinear potential.  Let us denote $\mathcal{T}_{\alpha}^1 = \mathcal{I}_{\alpha}$, as is standard for the dyadic Riesz potential.  Let us continue using the notation that $\mathcal{Q}$ be the lattice of dyadic cubes in $\mathbf{R}^n$.  Our aim is to prove the following:
\begin{prop}\label{sharpprop}  Let $\epsilon>0$, then there exists $f\geq0$, along with constants $c_1, c_2>0$, depending on $\alpha$ and $n$, so that:
\begin{equation}\label{sharpst}
|\{x\, :\, \mathcal{I}_{\alpha}(f)(x)>2,\, \mathcal{M}_{\alpha}(f)(x)\leq\epsilon\}|\geq c_1 e^{-c_2/\epsilon} |\{x\, :\, \mathcal{I}_{\alpha}(f)(x)>1\}|.
\end{equation}
\end{prop}

\begin{proof} Let $\epsilon>0$.  To begin the proof, let us fix a cube $P\in \mathcal{Q}$.  Note that:
$$\mathcal{I}_{\alpha}(\chi_P)(x) \approx \ell(P_{x})^{\alpha-n}|P|,
$$
where $P_x$ is the smallest dyadic cube containing both $P$ and $x$.  Now, given $N\in \mathds{N}$ and $\delta>0$ to be chosen, let $Q^j = [0, 2^j)^n$ for $j=0, \dots, N$, and define $f(x)$ by:
\begin{equation}\label{defnf}
f(x) = \delta \chi_{Q^0} + \sum_{j=1}^N \delta \frac{\ell(Q^j)^{n-\alpha}}{|Q^j \backslash Q^{j-1}|}\cdot \chi_{Q^j \backslash Q^{j-1}}(x).
\end{equation}
Note that we can write:
$$Q^j \backslash Q^{j-1} = \bigcup_{\ell=1}^{(2^n-1)2^{n(j-1)}}Q^{j,\ell},
$$
where  $Q^{j, \ell} \in \mathcal{Q},\, \ell(Q^{j, \ell}) =1$, for all $\ell = 1, \dots, (2^n-1)2^{n(j-1)}$.  It therefore follows:
\begin{equation}\label{approxriesz}\mathcal{I}_{\alpha}(f)(x) \approx \delta \ell(Q^0_{x})^{\alpha-n} + \delta\sum_{j=1}^N  \frac{\ell(Q^j)^{n-\alpha}}{|Q^j \backslash Q^{j-1}|} \sum_{\ell=1}^{(2^n-1)2^{n(j-1)}} \ell(Q^{j, \ell}_{x})^{\alpha-n}.
\end{equation}
Let us define a auxiliary operator $\mathcal{A}$ by:
\begin{equation}\label{Adefn}\mathcal{A}(x)=  \delta \ell(Q^0_{x})^{\alpha-n} + \delta\sum_{j=1}^N  \frac{\ell(Q^j)^{n-\alpha}}{|Q^j \backslash Q^{j-1}|} \sum_{\ell=1}^{(2^n-1)2^{n(j-1)}} \ell(Q^{j, \ell}_{x})^{\alpha-n}.
\end{equation}
By (\ref{approxriesz}) it suffices to prove that, for suitable $c_1, c_2>0$:
\begin{equation}\label{ineq4a}
|\{x\, :\, \mathcal{A}(x)>2, \, \mathcal{M}_{\alpha}(f)(x) <\epsilon\}|\geq c_1 e^{-c_2/\epsilon}|\{\mathcal{A}(x)>1\}|.
\end{equation}

First we claim that, if $x\in Q^0$, then $\mathcal{A}(x) = (N+1)\delta$.  To see the claim, note that by construction $Q^{j,\ell}_x = Q^j$, for all $j\geq 1$.  Since $|Q^j \backslash Q^{j-1}| = (2^n-1)2^{n(j-1)}$ it follows:
\begin{equation}\label{claim1step}\mathcal{A}(x) = \delta + \delta\sum_{j=1}^N  \frac{\ell(Q^j)^{n-\alpha}}{|Q^j \backslash Q^{j-1}|} (2^n-1)2^{n(j-1)} \ell(Q^j)^{\alpha-n} = \delta (N+1),
\end{equation}
as required.

Our second claim concerns the decay of the potential in each annulus $Q^k\backslash Q^{k-1}$.  We will see the decay is approximately $\delta$ at each step.  In particular we claim that, if $x\in Q^k \backslash Q^{k-1}$, then:
\begin{equation}\label{claim2}
\mathcal{A}(x) = \delta \Bigl( \sum_{j=1}^k 2^{-j(n-\alpha)} + \sum_{j=1}^{k-1} 2^{-j\alpha} + \frac{2^n-2+2^{-k\alpha}}{2^n-1} +(N-k-1)\Bigl).
\end{equation}

To prove this claim, let us break up the sum in as $\mathcal{A}(x) = I + II + III$.  Here:
$$I = \delta \ell(Q^0_x)^{\alpha - n} + \delta \sum_{j=1}^{k-1}\frac{\ell(Q^j)^{n-\alpha}}{|Q^j \backslash Q^{j-1}|} \sum_{\ell=1}^{(2^n-1)2^{n(j-1)}} \ell(Q^{j, \ell}_{x})^{\alpha-n},
$$
$$II = \delta \frac{\ell(Q^k)^{n-\alpha}}{|Q^k \backslash Q^{k-1}|} \sum_{\ell=1}^{(2^n-1)2^{n(k-1)}} \ell(Q^{k, \ell}_{x})^{\alpha-n}, \text{ and:}
$$
$$III = \delta \sum_{j=k+1}^N  \frac{\ell(Q^j)^{n-\alpha}}{|Q^j \backslash Q^{j-1}|} \sum_{\ell=1}^{(2^n-1)2^{n(j-1)}} \ell(Q^{j, \ell}_{x})^{\alpha-n}.
$$
To estimate $I$, note that $Q^{j, \ell}_x = Q^k$ for any $j = 1, \dots, k-1$ and $Q^0_x = Q^k$.  Thus:
$$I = \delta 2^{k(\alpha - n)} \sum_{j=0}^{k-1}2^{j(n-\alpha)} = \delta \sum_{j=1}^k 2^{(\alpha -n)j}.
$$
Now, we move onto estimate $II$.  Let us first make some elementary observations.  Note that there is one cube in $\{Q^{k, \ell}\}_{\ell}$, such that $\ell(Q^{k, \ell}_x) = 1$, and for each $j= 1, \dots k-1$, there are $(2^n-1)2^{n(j-1)} (= |Q^j\backslash Q^{j-1}|)$ cubes in the collection $\{Q^{k, \ell}\}_{\ell}$, such that $\ell(Q^{k, \ell}_x) = 2^j$.  When $j=k$, there are $(2^n-2)2^{n(k-1)}$ cubes in $\{Q^{j,\ell}\}_{\ell}$ with  $\ell(Q^{k, \ell}_x) = 2^k$.  With this in mind we see that:
$$II = \delta \frac{2^{-k\alpha+1}}{2^n-1} + \delta \sum_{j=1}^{k-1} 2^{\alpha (j- k)} + \delta \frac{2^n-2}{2^n-1}.
$$
Finally, to estimate $III$, note that, whenever $j>k$, then: $Q^{j, \ell}_x = Q^j$.  Thus, similarly to (\ref{claim1step}), we see that:
$$III = (N-k-1)\delta, \text{ and hence }(\ref{claim2})\text{ is proved}.
$$
%If $x\in Q^k \backslash Q^{k-1}$, then (\ref{approxriesz}) gives:
%\begin{equation}\begin{split}\mathcal{A}(x) & = \sum_{j=1}^{k-1} \delta \frac{\ell(Q^k)^{n-\alpha}}{|Q^k\backslash Q^{k-1}|}\frac{|Q^j\backslash Q^{j-1}|}{\ell(Q^j)^{n-\alpha}} + \delta\sum_{j=0}^k \frac{|Q^j\backslash Q^{j-1}|}{|Q^k\backslash Q^{k-1}|}  + \sum_{j=k+1}^N \delta.
%\end{split}\end{equation}
%Therefore, if $x\in Q^k\backslash Q^{k-1}$,
%$$(N-k-1)\delta \lesssim I_{\alpha}(f)(x) \lesssim (N-k +C)\delta,
%$$
%$$\text{and, }I_{\alpha}(f)(x) \leq I_{\alpha}(f)(y), \text{ if } y\in Q_0.
%$$
We now turn to the size estimates for the fractional maximal operator.  If $x\in Q_0$,
\begin{equation}\label{maximalsmall}\mathcal{M}_{\alpha}(f)(x) = \sup_{0\leq k\leq N} \frac{\delta }{2^{k(n-\alpha)}} \sum_{j=0}^k2^{j(n-\alpha)}\leq \frac{\delta}{1-2^{\alpha-n}} .
\end{equation}
Now, let us choose: $$\delta = \epsilon(1-2^{\alpha-n}), \text{ and }N = \Bigl\lfloor \frac{2}{\delta} \Bigl\rfloor.$$
It follows from (\ref{claim1step}), (\ref{claim2}) and (\ref{maximalsmall}) that: $\{\mathcal{A}(x) >2, \, \mathcal{M}_{\alpha}(f)(x)\leq \epsilon\} \supset Q^0$, and $\{\mathcal{A}(x) >1\}\subset Q^{k_0}$ with $k_0 \approx 1/\epsilon$.  The proposition follows.
\end{proof}

\subsection{The proof of Theorem \ref{contthm}}  In this subsection we show how the dyadic proof above can be modified to obtain Theorem \ref{contthm}.

\begin{proof}[Proof of Theorem \ref{contthm}]
We will again assume $\mu$ is compactly supported.  It will be most convenient to work with the pointwise equivalent operator:
$$\mathbf{F}(\mu) = \Bigl(\sum_{j\in \mathds{Z}}\Bigl(\frac{\mu(B(x,2^j))}{2^{j(n-\alpha)}}\Bigl)^q \Bigl)^{1/q}.$$
Let us denote by $G_{\lambda} = \{x\, :\, \mathbf{F}(\mu)(x)>\lambda\}$.  From the Whitney cube decomposition, one obtains disjoint dyadic cubes $\{Q_j\}$, so that:
\begin{equation}\begin{split}\label{whitprop}(i). \,\bigcup_j Q_j = & G_{\lambda},\, (ii).\, \sum_j \chi_{2Q_j} \leq C(n)\chi_{G_{\lambda}}, \text{ and }\\
&(iii).\,\frac{3}{2}<\frac{\text{dist}(Q_j, \mathbf{R}^n)}{\text{diam}(Q_j)}\leq 6.\end{split}\end{equation}
Arguing as in (\ref{goodlambdalocal}), we fix such a cube $Q_j$ with side length $2^{j_0}$, and we may suppose there exists $z\in Q_j$ so that $\mathbf{M}_{\alpha}(\mu)(z)\leq \epsilon\lambda$.  Let us denote: $S = \{x\in Q_j\, :\, \mathbf{F}(\mu)(x) > \tau\lambda, \; \mathbf{M}_{\alpha}(\mu)(x) \leq \epsilon \lambda\}$.
As in (\ref{tailest}), we see by definition of the Whitney cube, that there exists a constant $\tau>1$ so that, if $x\in Q_j$ is such that $M_{\alpha}(\mu)(x)\leq \epsilon \lambda$:
\begin{equation}\label{conttailest}
\sum_{j>j_0} \Bigl(\frac{\mu(B(x,2^j))}{2^{j(n-\alpha)}}\Bigl)^q \leq \Bigl(\frac{\tau}{2}\Bigl)^q\lambda^q. 
\end{equation}
Indeed, by the property (iii) of the Whitney decomposition, we find $y\not\in Q_j$ so that $d(y, Q_j)\leq 6\sqrt{n} \ell(Q_j)$, and $\mathbf{F}(\mu)(y)\leq \lambda$.  Therefore, with $x\in Q_j$ such that $M_{\alpha}(\mu)(x)\leq \epsilon \lambda$, it follows, for a constant $A>0$ depending on dimension:
\begin{equation}\begin{split}
\sum_{j>j_0} \Bigl(\frac{\mu(B(x,2^j))}{2^{j(n-\alpha)}}\Bigl)^q & \leq \sum_{j_0+A \geq j>j_0}  \Bigl(\frac{\mu(B(x,r))}{2^{j(n-\alpha)}}\Bigl)^q + \sum_{j>j_0 + A} \Bigl(\frac{\mu(B(x,2^j))}{2^{j(n-\alpha)}}\Bigl)^q\\
& \leq A\epsilon \lambda + C(n,q,\alpha)\sum_{j>j_0} \Bigl(\frac{\mu(B(y,2^j))}{2^{j(n-\alpha)}}\Bigl)^q \leq \Bigl(\frac{\tau}{2}\Bigl)^q \lambda^q.
\end{split}\end{equation}
Let us now, for $j\leq j_0$, define:
$$g_k(x) = 2^{j(\alpha-n)}\mu(B(x,2^j)).
$$
Then, by Fubini's theorem and since $k\leq j_0$, it follows for any $z\in Q_{j}$:
\begin{equation}\label{contgint}\begin{split}
\int_{Q_j} g_k(x) dx &= \int_{Q_j+B(0,2^k)} 2^{k(\alpha-n)}|B(y, 2^k)| d\mu(y) \\
&= 2^{k\alpha}|(Q_j)+2^k|_{\mu}\leq c(n)2^{(k-j_0)\alpha}|Q_{j}|\mathbf{M}_{\alpha}(\mu)(z).
\end{split}\end{equation}
With $z$ chosen as above, so $\mathbf{M}_{\alpha}(\mu)(z)\leq \epsilon \lambda$.
Hence, as in the proof of (\ref{glevel}):
\begin{equation}\label{gintsmall}
|\{x\in Q_j\, :\, g_k >\gamma\}|\leq \frac{c(n)}{\gamma}|Q_{j}|2^{k-j_0}\epsilon \lambda
\end{equation}
The rest of the proof therefore follows as in the dyadic case.  Indeed, for $x\in Q_j$ so that $\mathbf{M}_{\alpha}(\mu)(x)\leq \epsilon \lambda$:
$$\sum_{j\in \mathds{Z}} \Bigl(\frac{\mu(B(x,2^j))}{2^{j(n-\alpha)}}\Bigl)^q \leq \Bigl(\frac{\tau}{2}\Bigl)^q \lambda^q + m(\epsilon \lambda)^q + \sum_{j\leq j_0-m}\Bigl(\frac{\mu(B(x,2^j))}{2^{j(n-\alpha)}}\Bigl)^q
$$
Let $\beta<\alpha$.  If $x\in S$, then $\mathbf{F}(\mu)(x) > 2\lambda$, and hence
$$S\subset Q_j\cap E, \text{ where:}$$
$$E = \bigcup_{k\leq j_0-m}^{\infty}\{g_k\geq \lambda 2^{\beta(k-j_0+m)}((\tau^q - m\epsilon^q -(\tau/2)^q) (1-2^{-q\beta}))^{1/q}\}.
$$
Using (\ref{gintsmall}). one estimates:
$$|S|\leq 2^{-\alpha m}|Q_j|\frac{c(n)\epsilon}{[(1-2^{-q\beta})(\tau^q - m\epsilon^q -(\tau/2)^q)]^{1/q}} \sum_{k=0}^{\infty} 2^{(\beta-\alpha)k}
$$
Now, let $m$ be:
$$m = \lfloor \frac{\tau^q(1-2^{-q})}{\epsilon^q}-1 \rfloor.
$$
Then, $|S|\leq C2^{-\alpha m}|Q_j|$, where $C = C(n,p,\alpha)>0$.   To complete the argument we recall a very standard trick.  Thus far, we have asserted the existence constants $\tau,c, C$, depending on $n, q$ and $\alpha$, such that:
$$| \{x\in Q_j\, :\, \mathbf{F}(\mu)(x) > \tau \lambda, \; \mathbf{M}_{\alpha}(\mu)(x) \leq \epsilon \lambda\}|\leq Ce^{-c\lambda^q}|Q_j|,
$$
for each Whitney cube $Q_j$.  Let $\sigma$ be a weak $A_{\infty}$-weight, then by the definition (\ref{defainf}), it holds that:
\begin{equation}\label{afterdefnainf}
| \{x\in Q_j\, :\, \mathbf{F}(\mu)(x) > \tau \lambda, \; \mathbf{M}_{\alpha}(\mu)(x) \leq \epsilon \lambda\}|_{\sigma}\leq Ce^{-c\lambda^q}|2Q_j|_{\sigma}
\end{equation}
for constants $c,C$ now in addition depending on the weak $A_{\infty}$ character.
Using property (ii) of the Whitney decomposition (see (\ref{whitprop})), we can sum (\ref{afterdefnainf}) to obtain (\ref{contlambda}).
\end{proof}

\subsection{The proof of Theorem \ref{goodtauthm}}
\begin{proof}[Proof of Theorem \ref{goodtauthm}]
Once again we assume that $\mu$ has compact support.  We split the set $\{(\mathbf{T}(\mu))^q(x) > \lambda\}$ into the Whitney cubes satisfying (\ref{whitprop}) as usual. Let us assume that cube $Q_j$ contains $x$ such that $\mathcal{M}_\alpha(\mu)^q(x) <\epsilon \lambda$. Let us put $r_0=\ell(Q_j$) and fix a point $y$ such that $(\mathbf{T}(\mu))^q(x)\leq\lambda$ and
$|x-y|\leq A r_0$, with $A=A(n)>0$. We write 
\begin{equation}\int_{r_0}^{\infty}\Bigl( \frac {\mu(B(x,r)}{r^{n-\alpha}} \Bigl)^{q}\frac{dr}{r}\leq \sum_{k=0}^{\infty}\int_{2^{k}r_0}^{2^{k+1}r_0} \Bigl( \frac {\mu(B(x,r)}{r^{n-\alpha}} \Bigl)^{q} \frac{dr}{r}.
\end{equation}
First we observe that since $\mathcal{M}_\alpha(\mu)^q(x) <\epsilon \lambda$ we have
$$\int_{2^{k}r_0}^{2^{k+1}r_0} \Bigl(\frac {\mu(B(x,r)}{r^{n-\alpha}}\Bigl)^q \frac{dr}{r}\leq \epsilon \lambda$$ for any $k$. Next, we see that 
\begin{equation}\begin{split}\int_{2^{k}r_0}^{2^{k+1}r_0}  \Bigl( \frac {\mu(B(x,r)}{r^{n-\alpha}} \Bigl)^{q}  \frac{dr}{r}& \leq 
(1+A2^{-k})^{(n-\alpha)q} \int_{(2^{k}+A)r_0}^{(2^{k+1}+A)r_0} \Bigl( \frac {\mu(B(y,r)}{r^{n-\alpha}} \Bigl)^{q}  \frac{dr}{r}\\
&\leq  \int_{(2^{k}+A)r_0}^{(2^{k+1}+A)r_0}  \Bigl( \frac {\mu(B(x,r)}{r^{n-\alpha}} \Bigl)^{q} \frac{dr}{r} +s(k)\epsilon \lambda.
\end{split}\end{equation}
Here $s(k)= (1+A2^{-k})^{(n-\alpha)q} -1  $ is a term with geometric decay in $k$. Summing the above, we obtain
\begin{equation}\begin{split}\nonumber\int_{r_0}^{\infty}\Bigl( \frac {\mu(B(x,r)}{r^{n-\alpha}}\Bigl)^q \frac{dr}{r}& \leq C_s \epsilon \lambda +\sum_{k=0}^{\infty} \int_{(2^{k}+A)r_0}^{(2^{k+1}+A)r_0}\Bigl( \frac {\mu(B(y,r)} {r^{n-\alpha}} \Bigl)^q\frac{dr}{r}\\
& \leq C_s \epsilon \lambda + (\mathbf{T}(\mu))^q(y) \leq (1+ C_s \epsilon) \lambda.\end{split}\end{equation}

We denote $\mu^j=\mu\chi_{3*Q_j}$ and note that we have
$\|\mu^j\|\leq C_n (\epsilon \lambda)^{1/q} r_0^{n-\alpha}$ from the fractional maximal operator estimate. Now     
\begin{equation}\begin{split}(\mathbf{T}(\mu))^q(x)& =\int_{0}^{r_0}\Bigl( \frac {\mu(B(x,r))}{r^{n-\alpha}} \Bigl)^q \frac{dr}{r} \leq \int_{0}^{r_0} \Bigl(\frac{\mu^j(B(x,r))}{r^{n-\alpha}}\Bigl)^q \frac{dr}{r} \\
&\leq C_{n,q} r ^{\alpha q } (M\mu^j(x))^q.
\end{split}\end{equation}

Here $M$ is the Hardy-Littewood maximal operator.  From the classical weak type estimate for $M$, we conclude:  
\begin{equation}\begin{split}|\{(\mathbf{T}(\mu^j))^q > c'\epsilon \lambda\}\cap Q_j|& \leq |\{C r_0^{\alpha} M\mu^j(x) > c'(\epsilon \lambda)^{1/q}\}|\\
&\leq  \frac{C_n (\epsilon \lambda)^{1/q} r_0^{n-\alpha}}{c' (\epsilon\lambda)^{1/q} C_{n,q}^{-1} r_0^{-\alpha}}\leq \frac{C_n}{c'C_{n,q}}|Q_j|. \end{split}\end{equation}
Now, fix $\sigma$ to be a weak-$A_{\infty}$ weight.  It immediately follows from the definition (\ref{defainf}) that:
$$|\{(\mathbf{T}(\mu))^q > (1+c'\epsilon) \lambda\}\cap Q_j|_{\sigma} \leq C(\sigma)\Bigl(\frac{C_n}{c'C_{n,q}}\Bigl)^{\theta}|2Q_j|_{\sigma}.
$$
Now, by summation of this inequality, using the property (ii) from (\ref{whitprop}), we deduce:
$$|\{(\mathbf{T}(\mu))^q > (1+c'\epsilon) \lambda\}\cap \{\mathbf{M}_\alpha \leq \epsilon \lambda\}|_{\sigma} \leq C_nC(\sigma)\Bigl(\frac{C_n}{c'C_{n,q}}\Bigl)^{\theta}|\{(\mathbf{T}(\mu))^q>\lambda\}|_{\sigma}.
$$
Picking $c'$ sufficiently large (in terms of $\alpha$, $q$, $n$ and the weak $A_{\infty}$ character of $\sigma$) proves the theorem.
\end{proof}

\section{Exponential integrability}\label{expsec}

We begin this section by deducing Theorem \ref{morremb} from Theorem \ref{contthm}.

\begin{proof}[Proof of Theorem \ref{morremb} from Theorem \ref{contthm}]
Let $B=B(x_0, R)$, and without loss of generality let us assume that:
\begin{equation}\label{maximal1}||\mathbf{M}_{\alpha}(\mu_B)||_{L^{\infty}(B)}\leq 1.\end{equation}
Note that for $x\not\in 2B$, it follows that:
$$\mathbf{T}(\mu_B)(x) \leq ((n-\alpha)q)^{1/q}\frac{\mu(B)}{R^{n-\alpha}}\leq c(q).
$$
Therefore, whenever $\lambda>c(q)$, it follows:
$$\{x\in \mathbf{R}^n\, :\, \mathbf{T}(\mu_B)(x)>\lambda\} \subset 2B.$$
For fixed $\lambda>c(q)$, letting $\epsilon = 1/\lambda$ in (\ref{contlambda}), we deduce that:
\begin{equation}\label{levelset}
|\{x\in 2B: \mathbf{T}(\mu_B)>2\lambda\}|_{\sigma}\leq e^{-c\lambda^q} |\{x\in 2B: \mathbf{T}(\mu_B)>\lambda\}|_{\sigma}.
\end{equation}
Here we are using the normalization (\ref{maximal1}).  Therefore, for all $\lambda>c(q)$:
$$|\{x\in 2B: \mathbf{T}(\mu_B)>2\lambda\}|_{\sigma}\leq Ce^{-c\lambda^q} |2B|_{\sigma}.
$$
The theorem follows from integrating this inequality.
\end{proof}

Let us now turn to proving Theorem \ref{morremb} as a result of Theorem \ref{goodtauthm}:

\begin{proof}[Proof of Theorem \ref{morremb} from Theorem \ref{goodtauthm}]
Assume that $\mathbf{M}_\alpha(\mu_B) \leq 1$ and proceed by induction. We first put $\lambda = c$ and we choose $\epsilon = 1/c$. This leads to an estimate
$$   
|\{x\in \mathds{R}^n\, :\, (\mathbf{T}(\mu_B))^q(x) > 2\lambda\}|_{\sigma} \leq C|\{x\in \mathds{R}^n\, :\,(\mathbf{T}(\mu_B))^q(x) > \lambda\}|_{\sigma}.
$$
Now at level $k$, we put $\epsilon=1/( 2^{k} c)$. For $1\leq l \leq 2^{k}$ we obtain inequalities 
\begin{equation}\begin{split}   
|\{x\in \mathds{R}^n\, :\, &(\mathbf{T}(\mu_B))^q(x)  >(1+l 2^{-k}) 2^k\lambda\}|_{\sigma} \\
&\leq C|\{x\in \mathds{R}^n\, :\, (\mathbf{T}(\mu_B))^q(x)> (1+(l-1) 2^{-k})2^{k}\lambda\}|_{\sigma}.
\end{split}\end{equation}
This proves the required exponentional decay of the distribution function.
\end{proof}

\begin{rem}\label{expsharp}To exhibit the sharpness of this inequality, let $\alpha = 1$, and consider $\mu = |x|^{-1}\chi_{B(0,1)}$.  Then, one can readily compute:
$$\mathbf{T}_1^q(\mu)(x) \approx (\log|x|)^{1/q}, \text{ for }x\in B(0,1).
$$
Therefore, $\mathbf{T}_1^q(\mu)$ is exponentially integrable to the $q$-th power, but not to any power greater than $q$.
\end{rem}

Let us now turn to Theorem \ref{quasthm}.   We will first define a suitable notion of solution:
\begin{defn}\label{vweak}  Let $\Omega$ be a bounded open set.  We say that $u$ is a \textit{very weak solution} of (\ref{dirprob}) if:
$$u\in L^{\frac{(p-1)n}{n-p},\infty}_{\text{loc}}(\Omega)\footnote{By definition, $f\in L^{s,\infty}_{\text{loc}}(\Omega)$ if for each compact set $K\subset \Omega$, there exists a constant $C_K>0$ so that 
$ \sup_{\lambda>0} \lambda|\{x\in K\, :\, |f(x)|>\lambda\}|^{1/s}\leq C_K$.}, \text{ and } |\nabla u|\in L^{\frac{(p-1)n}{n-1},\infty}_{\text{loc}}(\Omega),
$$
and equation (\ref{dirprob}) holds in the sense of distributions.
\end{defn}
That such a solution exists for each finite measure $\mu$ is not obvious, and the result can be found in \cite{DMMOP, HKM}\footnote{In \cite{DMMOP} the stronger notion of \textit{renormalized solution} is considered}.  There are many open problems around notions of solution to such quasilinear elliptic equations.  Our results for the equation (\ref{dirprob}) are a result of the following recent result of Duzaar and Mingione \cite{DM2}.
\begin{thm}[\cite{DM2} Theorems 1.1, 1.2 and Remark 1.1] \label{gradbounds} Suppose that $u$ is a very weak solution of:
$$-\Delta_p u = \mu \text{ in }\Omega
$$
for a finite nonnegative measure $\mu$.  Then:

\indent {\rm(i).} If $1-2/n <p<2$, then there exists a constant $C=C(n,p)>0$ so that for all balls $B(x,R)\subset \Omega$:
\begin{equation}\label{pleq2est}|\nabla u(x)|^{p-1}\leq C\mathbf{I}_1(\mu\chi_{B(x,R)}) (x)+ C\Bigl(\int_{B(x,r)}|\nabla u| dz\Bigl)^{p-1}.
\end{equation}
\indent {\rm(ii).} If $p\geq 2$, then  there exists a constant $C=C(n,p)>0$ so that for all balls $B(x,R)\subset \Omega$:
\begin{equation}\label{pgeq2est}|\nabla u(x)|^{p-1} \leq C \mathbf{T}_1^{\frac{p}{2(p-1)}}(\mu \chi_{B(x,R)})(x) + C \Bigl(\int_{B(x,r)}|\nabla u|^{\frac{p}{2}} dz\Bigl)^{\frac{2(p-1)}{p}}.
\end{equation}
\end{thm}

Let us now use Theorem \ref{gradbounds} and Theorem \ref{morremb} to deduce Theorem \ref{quasthm}.

\begin{proof}[Proof of Theorem \ref{quasthm}]
Let $B=B(x_0,R)$ be so that $2B \subset \Omega$.   Without loss of generality, let us suppose that:
$$||\mathbf{M}_{\alpha}(\mu_{2B})||_{L^{\infty}(2B)} =1.
$$
From Theorem \ref{gradbounds}, it follows that for all $x\in B(x_0,R)$:
\begin{equation}\label{gradestloc}|\nabla u(x)|^{p-1} \leq C\mathbf{T}(\mu\chi_{B(x_0,2R)})(x) + C\Bigl(\dashint_{B(x_0, 2R)}|\nabla u|^{\max(1,p/2)}dx\Bigl)^{\frac{p-1}{\max(1,p/2)}}.
\end{equation}
Here $\mathbf{T}$ denotes the relevant operator appearing in Theorem \ref{gradbounds}.  From Theorem \ref{morremb}:
\begin{equation}\label{expintdoub}\dashint_{2B}\exp\Bigl(c\frac{\mathbf{T}(\mu\chi_{2B})(x)}{||\mathbf{M}_{\alpha}(\mu\chi_{2B})||_{L^{\infty}(2B)}}\Bigl)^q dx\leq C.
\end{equation}
With $q=1$ if $p<2$ and $q = p/2(p-1)$ if $p>2$.  From (\ref{gradestloc}) and (\ref{expintdoub}) the result follows.
\end{proof}

\begin{rem}\label{gradregrem}  The exponential integrability result (\ref{pleq2bd}) in the case when $p\leq 2$ is sharp.  It seems reasonable to conjecture that one should also have integrability to the $p-1$ power in the bound (\ref{pgeq2bd}).  In lieu of Remark \ref{expsharp}, this bound could only be achieved by potential estimates if the estimate (\ref{pgeq2est}) was improved to match (\ref{pleq2est}).  In light of \cite{DM2,DM3}, this would be a deep result.
\end{rem}

\begin{proof}[Proof of Theorem \ref{uexpintthm}]The proof of Theorem \ref{uexpintthm} follows from combining Theorem \ref{morremb} with the potential estimate of Kilpel\"{a}inen and Maly (see \cite{KM}, Theorem 1.6), in an analogous way to the proof of Theorem \ref{quasthm}.  We leave the details to the reader.
\end{proof}

\end{document}